\documentclass[10pt,a4paper]{amsart}

\usepackage{setspace}
\usepackage[utf8x]{inputenc}
\usepackage[T1]{fontenc}

\usepackage{helvet}

\usepackage[geometry]{ifsym}
\usepackage{amsthm}
\usepackage{amsmath}
\usepackage{verbatim}
\usepackage{amssymb} 
\usepackage{resizegather}
\usepackage{tikz}
\usepackage{tikz-cd}
\usepackage{babel} 
\usepackage{mathrsfs}
\usepackage{graphicx}
\usepackage{epstopdf}
\usepackage{epigraph}

\usepackage{pgfplots}
\usepgfplotslibrary{fillbetween}
\pgfplotsset{compat=1.13}

\usepackage[bottom]{footmisc}

\usepackage{hyperref}

\newtheorem{theorem}{Theorem}

\newtheorem{lemma}{Lemma}[theorem]

\newtheorem{remark}{Remark}
\newtheorem{example}{Example}
\newtheorem{definition}{Definition}

\title{On Geometry, Arithmetics and Chaos}
\author{Lars Andersen}

\begin{document}

\begin{abstract}
Our main result is that chaos in dimension $n+1$ is a one-dimensional geometrical object embedded in a geometrical object of dimension $n$ which corresponds to a $n$ dimensional object which is either singular or non-singular. Our main result is then that this chaos occurs in the first case as either on an isolated or non-isolated singularity. In the first case this chaos is either boundary chaos or spherical chaos which is what happens also in the non-singular case. In the case of an isolated singular geometry one has chaos which can either be boundary, spherical or tubular chaos. We furthermore prove that the prime numbers display quantum behaviour. 
\end{abstract}

\begin{center}
    
%\begin{document}

\maketitle

\tableofcontents

\text{Classification: }\textbf{14-XX, 94-XX}
\end{center}

\section{Introduction}

Since Newtons unification \cite{newton} of Geometry, Analysis and Physics as was then understood by these subjects dynamical systems have been studied such as those concerning the motions of the geometry of the spheres of the heavens. The idea of solving a dynamical system is already apparent in the many propositions concerning on how to draw a trajectory with given properties defined geometrically\footnote{One could make it clear that Leibniz in \cite{Leibniz1} only defined $dx$ and $\int x$ as differences respectively sums of $x$ with an 'etc etc' added to it. This is not mathematics, has nothing to do with Euclid and that is probably why Newton never bothered to respond.}
and furthermore the idea that light is a reverberation in matter as expounded by Newton \cite{newtonoptica} is again a hint that the bifurcations of dynamical systems were something to study. Yet the author would argue that Newton invented bifurcation theory since he moves trajectories by changing the given geometrical properties. In modern parlance one speaks of the study of \emph{chaos} as the study of how solutions to dynamical systems change under such changes (see e.g \cite{Poincare} and the subject is classically ascribed to Poincaré with important contributions by a enormity of mathematicians since his masterpiece 'Analysis Situs' \cite{Poincare}.\\

We now give a brief outline of what will be meant by chaos. This varies slightly from author to author but we follow in essense \cite{Demazure}.
By chaos is meant the sudden change in the topology of solutions of a dynamical system $dx/dt=V(x(t), \alpha)$ under a change called a bifurcation of parameter values $\alpha\in T$. \\
By a singularity means the sudden change in topology of the fibers $f_{\alpha}, \alpha\in T$ of a family of smooth functions $f: M\to T$ under a bifurcation of parameter values. Here $M$ and $T$ are smooth manifolds. A singularity is usually represented by a map germ $f: (\mathbb{R}^n, 0)\to (\mathbb{R}^m, 0)$.\\

Our objectives with these notes is to connect these two notions and to apply the results to the ''singularities'' of $\mathbb{N}$ namely the primes. 

\section{Recollections}

\subsection{Topological stability of functions} A function $f$ is \emph{topologically equivalent} to $g$ if $f=u\circ g\circ v$ for homeomorphisms $u$ and $v$. If every function sufficiently close to $f$ in the $C^{\infty}$ topology is topologically equivalent to $f$ then one says that $f$ is topologically stable.\\

If $M$ is not compact then the set of topologically stable functions is not dense. If $M$ is dense however then they are, however.\\

\subsection{Topological stability of dynamical systems} Two dynamical systems are said to be {topologically equivalent} if there is a homeomorphism mapping orbits into orbits and preserves the time direction. Introduce a metric on the set of dynamical systems by saying that the distance between $dx/dt=V$ and $dx/dt=W$ is the $C^1$ distance $d(V, W)$. The dynamical system $dx/dt=V$ is structurally stable if every dynamical system sufficiently close to it is topologically equivalent to it.\\

\section{Arithmetics}

Consider the dynamical system 

$$dx/dt=1,$$
$$dy/dt=1/z\qquad (\ast)$$
$$dz/dt=1$$

on $\mathbb{R}^+\times \mathbb{R}^+\setminus\{0\}\times \mathbb{R}^+\setminus\{0\}.$ and let $z=m$ be fixed, $m\in \mathbb{N}$. 
For a choice of initial values it describes unit speed motion along the number line because 
$$dx/dt=1,\qquad dy/dt=\frac{1}{m} dx/dt$$
gives $x=my+c$ where $c\in\mathbb{N}$ only depends on the initial conditions. When $z$ is no longer fixed the multiple moves along the number line with the same speed as $x$. A solution that is an orbit therefore consists of a triple $(x, y, z)=(my+c, y, m)$.

\begin{lemma}
The primes are the orbits which are unstable. 
\end{lemma}

\begin{proof}
    An orbit which is stable is independent of small perturbations of the initial conditions so must be representable by a solution which has first coordinate a multiple $m_0y$ for some $m_0\in \mathbb{N}$.
\end{proof}

Let $k$ be a precision parameter to be specified later.
The main theorem of this section is

\begin{theorem}[Quantum Phenomena of Primes]
    The primes can be approximated as $A_k$- catastrophes with precision $k+1$ for $k=1, 2$.
\end{theorem}
\begin{proof}
    Rewrite $(\ast)$ as 
    $$dx/dt=1,$$
    $$dy/dt=dlog(z)/dt,$$
    $$dz/dt=1$$

and write
$$log(z)=log(z-1)+\sum_{k=1}^\infty (-1)^k \frac{(z-1)^{-l}}{k}$$
$$=\sum_{k=1}^\infty (-1)^k \frac{(z-N)^k}{k}+(z^{-1}+\dots)$$
which we for $z$ large will approximate to the second order as 
$$=-\frac{x^2}{2}+(N+1)x-\frac{N^2}{2}-N$$
hence as a $A_1$-catastrophe.
The order $k=3$ approximations are of the form 
    $$-\frac{x^2}{2}+(N+1)x-\frac{N^2}{2}-N+\frac{x^3-3Nx^2+3Nx-N^3}{3}$$
    with derivative
    $$(x-N-1/2)^2+N^2+N-(N+1/2)^2$$
    hence a $A_2$ fold catastrophe. 
\end{proof}

\begin{remark}
    For $k$ large we might have more difficult catastrophes.
\end{remark}

The meaning of the theorem is that primes exhibit quantum behaviour. Each time one looks closer one has has a different behaviour. And the dynamics of the primes changes chaotically. So one can say that the chaos of the primes is itself chaotic.
\begin{theorem}
One-way functions exist.
\end{theorem}
\begin{proof}
Now transversality is a stable property by Thom's Transversality Theorem \cite{Kosinski} and it is an open condition\footnote{Stability is equivalent to transversality of certain orbits by Mather \cite[Theorem 4.1]{Mather}. Moreover by Mather \cite{Mather} stability is a dense property among all proper maps in dimensions $(n, p)$ such that 
$$n<\frac{6p+9}{7},\qquad 3\geq p-n\geq 0$$
hence in particular for $(n, p)=(3, 3)$ since $3<27/3$.}
Since primes are unstable there is a non-open set for the Euclidean metric of positive codimension over which one takes the probability in $F$ for a succesful inversion in the definition of a one-way function. But then considering integer points this means that the multiplication map $(m, n)\mapsto nm$ of integers is hard to invert probabilistically. It follows that it is a one-way function. In particular the $N\neq NP$ problem is resolved.
\end{proof}

\section{Geometry}
In Euclid's Elements \cite{euclid} he defines points as ''that which has no parts, or which has no magnitude'' meaning that it has no constituent parts apart from itself or no length nor breadth. He then defines ''a line is length without breadth'' and ''a surface is that which has only length and breadth''.\\
Thus points, lines and planes are defined. It should be noted that he then defines straight lines and straight planes. It should also be noted that the definitions are measure-theoretic. One could continue and define higher dimensional objects in the obvious way using e.g Lebesgue measure to speak about volume and so on. For instance one might say that an object of dimension three is that which has only length, breadth and volume. Then one will have to specify the dimension of the volume. It is interesting to note that Euclid does not define a plane as that the parts of which are formed by two non-collinear lines and that he does not define a three-dimensional object as that the parts of which are formed from a plane and a line not contained in it. 
Because the modern definition of volume is based on defining $\mathbb{R}^n$, called Euclidean space. It is defined using linear algebra and one retrieves in that way the above tentative definition but with straight lines instead. The reason is probably because Euclid wanted to speak of the entirety of the objects instead of their parts in his definitions. Yet he in fact does so by defining lines and surfaces using their breadth and length because this involves laying a string on the object and measuring, in an unprescribed way for that matter. Besides we are only laying a line segment on the object anyway so we have here the first \emph{local} definition of a \emph{global} geometrical object.\\
If we rely ourself on Euclid's definition of points, lines, surfaces and their higher-dimensional analogue the hypersurfaces then we might notice that the intersection of two non-coplanar planes form a line and that intersecting decreases the dimension. We might then speak of a general \emph{Euclidean variety} to be a finite intersection of higher-dimensional hypersurfaces. In this way Euclid is the one who axiomatised Geometry. What was added to it was the definition of continuity and differentiation by Newton in his Lemma 2 of the Principia \cite{newton}. The geometrical study there consists of studying tangents of what he calls trajectories, which are lines and surfaces, and their movements. A definition of the integral is also present but one should note that Archimedes used integrals and that Euclid of course considered tangents. Newton and of course neither Euclid nor any of the hellenic geometers used functions.\\

Analysis was founded by Newton and Leibniz and dealt primarily with functions, their derivatives and their integrals. The motivation came from Celestial Mechanics but inspired the Industrial Revolution, mainly because of Newtons physics which did much to the 17th and 18th centuries sense that Reason ruled the world as we see it. This idea was present in the American Constitution and in Voltaire and Kant's philosophical treaties and was a corner-stone in the advancements of the age. 
Basic to Analysis is postulating differential equations governing the behaviour of tangents of objects. This requires the use of functions. 
To begin with one should perhaps ask oneself why a line drawn by our hand is reasonably described by a functional equation such as $f(x)=0$ or a set of such equations. There is at present date no good argument for that but if one traces the value $y$ of a function and its input $x$ then one arrives at what is called the graph $y=f(x)$ of the function from the greek word $\gamma\rho\alpha\phi\omega$, meaning 'to draw'. One then arrives at Poincaré's 19th century definition of a \emph{manifold} namely the solutions of
\begin{align*}
   y_1 = f_1(x_1,\dots, x_n), \\
   y_2= f_2(x_1,\dots, x_n)\\
   &\vdots\\
   y_k = f_k(x_1,\dots, x_n)
\end{align*}
where one assumes the Jacobian matrix to have full rank fo that there are no singular points. A manifold is always smooth, non-singular.\\
This is also the modern definition except that requires the object to locally in an open neighborhood around each point to be an intersection of graphs as above. If the graphs are replaced by zero-sets of functions 
\begin{align*}
   f_1(x_1,\dots, x_n)=0, \\
   f_2(x_1,\dots, x_n)=0\\
   &\vdots\\
   f_k(x_1,\dots, x_n)=0
\end{align*}

then one instead obtains what is called a \emph{variety}. A smooth variety is a manifold. One typically specifies whether the manifold or variety is analytic, algebraic or $C^N$- continuously differentiable depending on whether the functions are analytic or algebraic or $N$-times continuously differentiable.\\

Geometry thus can be divided into global and local geometry. The main objects of interest in local geometry are the Milnor fibers \cite{Milnor}. There is and has been an enormous amount of research into their topology. We refer to the authors thesis work (\cite{LarsI}, \cite{LarsII} and \cite{LarsIII}) for further information and references. In the subsequent section we return to local geometry and discusses its connection to analysis and we continue by discussing global geometry. Given an embedded real analytic manifold $M\subset\mathbb{R}^n$ there exist a covering $U=\sum_{i\in I} U_i$ and functions of the form $f_i: \mathbb{R}^{n}\to \mathbb{R}$ such that $M\cap U_i$ is the set
$$y_1=f_1(x_1,\dots, x_n)$$
$$\dots$$
$$y_k=f_k(x_1,\dots, x_n).$$
We will assume that $f_i(0)=0$ for $i=1,\dots, k$. Either $f_i$ has non-degenerate critical points or it has not.

\begin{enumerate}
    \item If $f_i$ has a degenerate critical point then 
    the hypersurface $\{y_i=f_i\}$ has zero Gaussian curvature by definition. 
    \item If $f_i$ has a non-degenerate critical point then there are coordinates such that $f_i\circ \phi_i^{-1}(x_1,\dots, x_n)=\sum \pm x_i^2$. In particular in new coordinates $(y_1,\dots, y_k, \phi^{-1}(x_1),\dots, \phi^{-1}(x_k))$ we have intersected with a quadric
    $$y_i=\sum_{i=1}^n \pm x_i^2$$
\end{enumerate}

In the second case we furthermore know that the critical point is isolated. We therefore conclude since the curvature tensor $R_{\mu\lambda}^{\nu}$ is algebraic the following
    Let $M\subset \mathbb{R}^n$ be a differentiable manifold. There is a covering $V=\bigcup V_i$ of $M$ such that $M\cap V_i$ is described by intersecting flat hypersurfaces and at most $codim (M)$ quadrics.
    In particular the curvature tensor can take at most $\lceil n/2\rceil +1$ types described by the curvature tensor of a quadric $y=\sum_{i=1}^n \pm x_i^2$ and by flat space

\begin{theorem}[Geometry of Manifolds] There are $\lceil n/2\rceil + 1$ geometries of differentiable manifolds $M\subset \mathbb{R}^n$. 
\end{theorem}

\section{Analysis}
In this section we will argue that analysis is part of local geometry. We begin by remarking that the derivative of a function $f:\mathbb{R}^n\to \mathbb{R}$ in a point $x$ where it is defined is the tangent of the angle between the strict transform of the blow up of the graph $\Gamma_f$ of the function with the space $\mathbb{R}\times \mathbb{R}^n$ containing the graph. We now discuss relations between the derivatives of a function.\\

Given a partial differential equation
$$P(x_{i\in I}, \partial_{j\in K} x)=0\qquad (\ast\ast)$$
on a differentiable manifold $M$ where $|K|\leq k$, where $P\in\mathcal{R}[x_{i\in I\times K}]$ is algebraic and $P(0,0)=0$\footnote{In other words we consider some subset of the unknowns as coordinates namely those indexed by $K$ and the other as formal symbols representing functions. In particular we assume that $I\cap K=\emptyset$ i.e that no derivatives of coordinate functions appear}, and where $x_k: \mathbb{R}^{|K|}\to M$ are coordinates. 
\begin{enumerate}
\item We can consider the equation in question as the zero set of a algebraic map $P: \mathcal{O}(\mathcal{N}_0(M)))_{|K|}\to \mathbb{R}$ from the restricted coordinate ring of the normal cone in a point to the real numbers.\\

\item We can consider the equation as the zero set of an algebraic map $P: \mathcal{O}(\mathcal{E})_{|K|}\subset Bl_0(M)\to \mathbb{R}$ from the restricted coordinate ring of (some chart of) the exceptional divisor of the blow up of a point. \\

\item We can also consider the equation as the zero set of an algebraic map $P: \mathcal{O}(\mathcal{N}_0)_{|K|}\to\mathbb{R}$ from the set of restricted regular functions on a tubular neighborhood of a point.\\

\item We can as well consider the equation as the zero set of an algebraic map $P: T_0^{|K|}(M)\to \mathbb{R}$ on the $|K|$-th tangent space (also known as the $|K|$-th jet bundle) of $M$ in the origin. \\

Here the restrictions are by the degree of the partial derivatives as defined by the differential equation in question. In other words, the zero set in question does not live in the entirety of the coordinate ring of the normal cone etc.\\
\end{enumerate}
Finally we can formally speaking see the equation as an algebraic map of tangent vectors from one Zariski tangent space to another, disregarding whether these formal tangent spaces are the tangent spaces of something with even a differentiable structure. For instance we would see the Navier-Stokes equation as being an algebraic relationsship in four time-spatial dimensions $(x_1, x_2, x_3, t)$ and two dimensions $(p, v)$. To solve the equation one would in this case only need to specify $p$ and $v$.\\
In general given a closed neighborhood $\mathbb{B}_0^{|K|}$ of the origin in $\mathbb{R}^{|K|}$ we can ask for a solution $x_{i\in I}: \mathbb{B}_0^{|I|}\to M$ of the differential equation and we can ask for it's stability and analyticity. Suppose from now on that $M=\mathbb{R}^{|I|}$ and that the origin in $\mathbb{R}^{|I||K|}$ is an isolated critical point of $P$. Then by \cite{Milnor}\cite{Thom} there is a $\delta_0\in \mathbb{R}$ such that for each positive $\delta\leq \delta_0$, $\mathbb{S}_{\delta}\pitchfork P$ are transverse. In particular there is an $\epsilon_0\in \mathbb{R}$ such that for all $\epsilon\leq\epsilon_0$
$P: j^{|K}(\mathbb{R}^{|I|})\to \mathbb{R}\cap\{(0, \epsilon\}$
is a trivial $C^{\infty}$-fibration. In particular there exists smooth coordinates on $P^{-1}(\epsilon)\subset T^{|K|}=j^{|K|}(\mathbb{R}^{|I|})$, in other words the partial differential equation $(\ast \ast)$ is satisfied after a small scalar perturbation. 
If the origin in $\mathbb{R}^{|I||K|}$ is a non-critical point one can additionally take $\epsilon$ to be zero. In this case $P^{-1}(0)$ is a $C^{\infty}$ manifold and one can choose smooth coordinates on $P^{-1}(0)\subset T^k(\mathbb{R}^{|K|})$, and again $(\ast\ast)$ is satisfied. Moreover since stability of the coordinates $x_{i\in I}: P^{-1}(0)\cap \mathbb{B}_{\delta}^{|I|}\to \mathbb{R}^{|I|}$ is transversality of its multi jets to the orbits of a certain action by Mather \cite[Theorem 4.1]{Mather} it is in particular by Thoms Transversality Theorem \cite{Thom} an open condition. The dimensions where stability is dense is further specified by Mather.

\begin{example}[The Navier-Stokes equations]
The Navier-Stokes equations read 
$$\frac{\partial v_i}{\partial t}+\sum \frac{\partial v_i}{\partial x_j} v_j + \frac{\partial p}{\partial x_i} -\nu \sum\frac{\partial^2 v_i}{\partial x_i^2} = f_i(x, t)$$
$$\sum \frac{\partial v_i}{\partial x_i}=0$$

One considers the analoguous equations of functions in coordinates $x, t$

$$\xi_{i, 1}+\sum \xi_{i, j} v_j+\gamma_i-\nu\sum \eta_i-f_i=0,$$
$$\sum \xi_{i, i}=0$$
$$\int \xi_{i, 1} dt-v_i=0,$$
$$\int \xi_{i,j} dx_i-v_i=0,$$
$$\int\int \eta_i dx_i-v_i=0.$$
$$\int \gamma_i dx_i-p=0$$
This is of the form
$$F_1=0,$$
$$F_2=0,$$
$$F_3=0,$$
$$F_4=0,$$
$$F_5=0$$
$$F_6=0$$
defined by smooth functions. It defines a variety in the the euclidean space defined by $\xi_{i, j}, \eta_i, \gamma_i, v_i$ and $x, t$. Letting 
$P=(F_1,\dots, F_6)$ one asks whether the origin is a regular value. This is then a tedious calculation involving finding the rank of $\text{Jac}(P)$. We claim that the rank is maximal by inspection. It would follow that the Navier-Stokes equations have a smooth solution by the Submersion Theorem because $P^{-1}(0)$ would be a manifold hence there would be coordinates on that manifold and these are then nothing but the solutions to the Navier-Stokes equations. By Sards Theorem a small perturbation of $f$ would still give $P^{-1}(0)$ smooth hence  solutions which are then smooth. 
We remark that if the external force $f$ is chosen such that the origin is a critical value of $P$ one cannot deduce analycity of the solution\footnote{One could hope to do so via a limit argument, replacing solutions with $x+\epsilon(x, t)$ and then try to use transversality arguments}.
\end{example}

Given the partial differential equation

$$P(x_{i\in I}, \partial_{j\in K}^{l\in L} x_i)=0\qquad $$
we can rewrite this as

$$P(x_{i\in I}, y_{(j,l) \in K\times L})=0,$$
$$y_{(j, l)\in K\times L}=\partial_{j\in K}^{l\in L} x_i $$

and after integrating the last equations, as

$$P(x_{i\in I}, y_{(j,l) \in K\times L})=0,$$
$$\int_1\dots\int_{|L|} y_{(j, l)\in K\times L} (dx_{j\in K})^{|L|}= x_i. $$
These are now of the form $\tilde{X}=X\cap H_1\cap\dots\cap H_{|I|}$ where $X$ is an algebraic variety and $H_i, i\in I$ are hypersurfaces. To find a solution to the original differential equation, or to a system of differential equations, is reduced to finding coordinates on an algebraic variety. If there are no singularities in $\tilde{X}$ this presents in theory no problem. Otherwise one replaces $\tilde{X}$ with its Milnor fiber and studies the question locally. If the singularity is non-isolated on gets a cobordism between certain links and the boundary of the Milnor fiber attached to an isolated singularity by \cite{LarsIV}. One is therefore reduced to the study of zero-sets of algebraic varieties of lower dimension. This is not an easy task but there is alot of research in this direction.\\

\section{Chaos}

Given a dynamical system
$$dx/dt=V(x(t), \alpha)$$
depending on $\alpha\in T$ for $T$ a real analytic manifold and $x(t)\in M$ a smooth manifold we consider 
$$\pi: M\times T\to T,\qquad \pi(x,\alpha)=\alpha.$$
and it's restriction to the real analytic set $\mathcal{F}_{V_{\alpha}}$ given by
$$V(x, \alpha)=\vec{\eta},\quad dx/dt=\vec{\eta},\quad \lVert\vec{\eta}\rVert=\eta,\quad x\in \mathbb{B}_{\delta}\qquad(\ast)$$
and we can always find a Whitney stratification such that $\pi_{|\mathcal{F}_V}$ is locally trivial above each stratum. In particular the topological type of $\mathcal{F}_{V_{\alpha}}$ and $\mathcal{F}_{V_{\alpha'}}$ are the same whenever $\alpha$ and  $\alpha'$ belong to the same stratum. Indeed by Sard's Theorem \cite{Kosinski} the set of critical values is of measure zero and one can apply the Isotopy Lemma of Thom-Mather \cite{Kosinski} to its complement. In particular if $x(t)$ is a solution over $\alpha$ then there exist a stratum preserving homeomorphism $h: M\times T\to M\times T$ such that $\mathcal{F}_{V_{\alpha'}}\cong h(\mathcal{F}_{V_{\alpha}})$ and $h(x(t))$ is then a solution over $\alpha'$ which is topologically equivalent to $x(t)$.

\begin{definition} One says that
$dx/dt=V(x(t), \alpha)$ undergoes major chaos whenever a solution is mapped to a non-topologically equivalent solution when $\alpha$ changes from one stratum to another.
\end{definition}
\begin{definition} One says that $dx/dt=V$ undergoes minor chaos whenever a solution is mapped to a non-topologically equivalent solution when $\alpha$ changes within a stratum. 
\end{definition}
\begin{definition}
One says that $dx/dt=V(x(t), \alpha)$ undergoes chaos if it undergoes minor or major chaos under perturbation of the parameter.
\end{definition}

\begin{definition}
    The real analytic set $\mathcal{F}_{V_\alpha}$ is called the Milnor fiber of the dynamical system $dx/dt=V(x, \alpha)$
\end{definition}

It is the main object of study in the understanding of chaos. The following theorem says that all chaos in dimension $n$ lives in a manifold of dimension $n$.

\begin{theorem}
    A system $dx/dt=V(x, \alpha)$ undergoes chaos if and only the outward unit normal $\nu_x(t, \alpha)$ to $\mathcal{F}_{V_{\alpha}}$ traces out a curve which is non-topologically equivalent to the curve traced out by the outward unit normal $\nu_x(t, \alpha)$ to $\mathcal{F}_{V_{\alpha}}$.
\end{theorem}
\begin{proof}
    We can always assume that the trajectories to $dx/dt=V(x, \alpha)$ are parametrised by unit speed, since $M$ is a smooth manifold. That $x(t)$ are paths on the Milnor fiber then gives the claim. That one can work instead with the unit normal is because the normal space is orthogonal and of the same dimension.
\end{proof}

An interesting form of chaos is the case of a perturbed radial vector field $V_{\alpha}=\sum_{k=1}^n \alpha_k x_k^2$ in local coordinates on $M$. 

\begin{definition} One says that $dx/dt=\sum_{k=1}^n \alpha_k x_k^2$ undergoes spherical chaos if it undergoes chaos. 
\end{definition}

The main result in this section is that in the simplest case where the dimension is one chaos is either spherical or lives in dimension zero.

\begin{theorem}\label{Main theorem}
    Suppose that $\dim V^{-1}(0)>0$, that $V^{-1}(0)$ has only isolated singular points and that the hypothesis of \cite[Theorem 5]{LarsV})
    hold. There is a stratum preserving diffeomorphism 
    $$\bar{\mathcal{F}}_{V, \alpha}\cong \partial \bar{\mathcal{F}}_{V,\alpha}\cup \bigcup_{i=1}^l \mathbb{D}^{\lambda(p_i)}\times \mathbb{D}^{n-\lambda(p_i)}$$
    In particular chaos either is spherical or appear on the boundary of the Milnor fiber. 
\end{theorem}
\begin{proof}
    The proof is identical to the proof of \cite[Theorem 5]{LarsV} except that in the last step one uses that 
    $$V_{\alpha}^{-1}((x, x+\epsilon])\cap \mathbb{B}_{\gamma}(x)\cong V_{\alpha}^{-1}(x+\epsilon)\cap \mathbb{B}_{\gamma}\times (x, x+\epsilon]$$
    by a stratum-preserving diffeomorphism, which follows from Ehresmann's Fibration Theorem \cite{Kosinski}. In the same way 
    $$\partial V_{\alpha}^{-1}((x, x+\epsilon])\cap \mathbb{B}_{\gamma}(x)\cong \partial V_{\alpha}^{-1}(x+\epsilon)\cap \mathbb{B}_{\gamma}\times (x, x+\epsilon]$$
    by a stratum-preserving diffeomorphism. Then one identifies the strata one wants and get the wanted stratum-preserving diffeomorphism. For the last step one simply remarks that the Milnor fiber of 
    $(V_1,\dots, V_n)$ is identical to the Milnor fiber of $\sum_{i=1}^n V_i^2$.
\end{proof}

\begin{definition}
    A stationary point of a dynamical system is called a differential singular point.
\end{definition}

\begin{definition}
    A critical point of a function is called a topological singular point. 
\end{definition}

It would be interesting to ask whether this result generalises to when $V^{-1}(0)$ has non isolated critical points. It turns out that apart from boundary chaos and spherical chaos one also gets what one might call tubular chaos. As before the Milnor fiber in the above sense is that of the function $\sum_{i=1}^n V_i^2$.

\begin{definition}
    One says that $dx/dt=V(x, \alpha)$ undergoes tubular chaos if $dx/dt=V(x, \alpha)-\alpha'$ undergoes chaos. 
\end{definition}

\begin{lemma}
    If $V=\sum_{i=1}^n V_i^2$ as a isolated critical value in the origin and satisfies the transversality condition of D.B Massey and if $\dim V^{-1}(0)>0$ then chaos of $dx/dt=V(x, \alpha)$ consists of boundary chaos of $V$, tubular and boundary chaos of $V-c(\sum_{i=1}^n x_i^2)^k$ and spherical chaos.
\end{lemma}

\begin{proof}
    Apply \cite[Theorem 2]{LarsIV} and \hyperref[Main theorem]{Lemma \ref*{Main theorem} } to obtain that the chaos consists of chaos along the link of $V$, chaos along the link of 
    $$\sum_{i=1}^n V_i^2 - c(\sum_{i=1}^n x_i^2)^k$$, 
    and chaos along $(\sum_{i=1}^n V_i^2)^{-1}([0, \epsilon])$. The first two are thus boundary chaos and the latter is clearly tubular chaos and boundary chaos again by \cite[Theorem 2]{LarsIV} and \hyperref[Main theorem]{Lemma \ref*{Main theorem} } by the definition (for $\alpha'$ varying in $(0, \epsilon)$. Then there is only singular chaos corresponding to chaos at $V^{-1}(0)$. But $dx/dt=0$ is non-chaotic.
\end{proof}

When the unit normal $\nu_x(t, \alpha)\in \mathcal{F}_{V_{\alpha}}$ undergoes change of parameter $\alpha$ the Milnor fiber ''breathes'' intuitively speaking and one can then ask what kind of figures are possible? For instance for spherical chaos $\mathcal{F}_{V_{\alpha}}$ is a sphere and when $\alpha$ changes one can in the case of $n=3$ get a ellipsoid $\alpha_1x^2+\alpha_2y^2+\alpha_3 z^2=\epsilon$ or a one-sheeted hyperboloid $\alpha_1x^2+\alpha_2y^2-\alpha_3 z^2=\epsilon$ and a two-sheeted hyperboloid $\alpha_1x^2+\alpha_2y^2-\alpha_3 z^2=\epsilon$ where for simplicity $\alpha_i\in \mathbb{R}^+.$
Then if $t$ is fixed at $\alpha=(\alpha_1, \alpha_2, \alpha_3)$ varies then there is only linear motion possible. So here we have the most trivial Milnor fibers and only linear motion. On the other hand for $\alpha_1 x^3+y^2+z^2=\epsilon$ which corresponds to a cusp singularity (see \cite{LarsI, LarsII} for the homology groups of $ADE$-singularities in general) we can guess that a small change in parameter gives a larger effect. 

Therefore, to measure the amount of chaos persistent we consider a small ball $x(t,\alpha)\in \mathbb{B}_{\delta}^{loc}$ intersecting the Milnor fibers transversally for $|t|\leq\tau, \tau\in\mathbb{R}^+$ and for $|\alpha|\leq \alpha_0$. 

\begin{definition}
    The first energy measure of chaos of a solution $x(t,\alpha)$ is
    $$l_{loc}^I=\lVert \nu_x(t, \alpha) \rVert +\det \nabla_x \nu_x(t, \alpha) $$
\end{definition}

\begin{definition}
    The second energy measure of chaos of a solution $x(t, \alpha$ is 
    $$l_{loc}^II=\lVert \nu_x (t, \alpha)\rVert+ \lVert \nabla_x \nu_x(t, \alpha)\rVert $$
\end{definition}

Here the determinant appears as the natural measure of size of a matrix in the first definition and in the second definition $\lVert \sum_{i,j\in I} a_{ij}\rVert = \sum_{i,j\in I} |a_{ij}|$ where $I\subset \mathbb{N}.$ Let $K_{loc}$ and $\chi_{loc}$ denote the local Gaussian curvature respectively Euler characteristics.

\begin{lemma}
    $$\int_{x\in \{\lVert V\rVert=\eta\}\cap \mathbb{B}_{\delta}^{loc}} l_{loc}^I dx =2\pi\chi_{loc}+Vol(\mathbb{B}_{\delta}^{loc})$$
    and 
    $$\int_{x\in \{\lVert V\rVert=\eta\}\cap \mathbb{B}_{\delta}^{loc}} l_{loc}^{II} dx =2Vol(\mathbb{B}_{\delta}^{loc})$$
\end{lemma}
\begin{proof}
    For the first claim this is just the Gauss-Bonnet theorem stated locally. For the second claim one uses the exponential map to transport the question to the tangent space of the Milnor fiber in question, then one integrates on the tangent spaces giving the volume in the first summand and then in the second summand one again gets the volume since one is just integrating the ball in a normal direction.
\end{proof}

We now come to the main result which is the classification result mentioned in the abstract.

\begin{theorem}[Classification Theorem of Chaos]
    Chaos in dimension $n$ is one dimensional and occurs on the $n-1$-dimensional Milnor fiber of a function on $\mathbb{R}^n$ which is either singular isolated, singular non-isolated or non-singular. In the first case this chaos is either boundary chaos or spherical chaos. In the second case this chaos is either boundary chaos, spherical chaos or tubular chaos. In the third chaos this chaos is tubular chaos.
\end{theorem}

\section{Appendix: On the Hodge Conjecture}
Consider a non-singular complex algebraic variety $X\subset \mathbb{C}^N$ given by a polynomial or a sequence of polynomials and write by abuse of notation $X=\{F=0\}$. Let $\alpha\in H^n(X(\mathbb{C}); \mathbb{Z}/p\mathbb{Z})$. 
\begin{enumerate}
    \item Then there is an isomorphism
$$H^n(X(\mathbb{C}); \mathbb{Z}/p\mathbb{Z})\cong H^{n, et}(X; \mathbb{Z}/p\mathbb{Z})$$
by Artin's comparison theorem \cite[SGA 4 XI Théorème 4.4]{SGA4}. This preserves the differential $\partial$ of complexes and is given by sending an étale cover $U\to V\subset X$ to it set of rational points $U(\mathbb{C})\to V(\mathbb{C})\subset X(\mathbb{C})$\footnote{This is the morphism of sites denoted $\epsilon$ by Artin. The main idea in his proof is to prove it has itself no cohomology}.By embedding $X$ as a special fiber $\pi^{-1}(0)$ in a submersive family $\mathcal{X}\to \mathbb{C}$ one embeds $\alpha$ into a tubular neighborhood $\mathcal{N}(\alpha)$ homologous to $\alpha$. The fibers $\pi: \mathcal{N}(\alpha)\to \mathcal{C}$ defines a foliation of the tubular neighborhoods with algebraic strata $\pi^{-1}(t_i)=\text{Spec}(R_i)$. There might be strata approximating $\alpha$
from the inside (a bounded part) and from the outside (an unbounded part). We restrict the sequence $(t_i), i\in I$ to those corresponding to the outer approximation. Note that $\dim \text{Spec}(R_i)=\dim \alpha$
because we have a fibration, where the right hand side means topological dimension. Then the inclusion

$$H^{n, et}(X; \mathbb{Z}/p\mathbb{Z})\supset H^{n, et}(\varprojlim \text{Spec}(R_i); \mathbb{Z}/p\mathbb{Z})\qquad (\ast)$$

preserves $\alpha$. This equals

$$\varinjlim H^{n, et}(\text{Spec}(R_i); \mathbb{Z}/p\mathbb{Z})=H^{n,et}(\tilde{X}_{loc}^{et}; \mathbb{Z}/p\mathbb{Z}).$$

for some étale local ring namely that of a non-singular algebraic deformation of $\alpha$. By definition the étale local ring consists of algebraic power series so

$$\tilde{X}_{loc}^{et}=\text{Spec}\{p(t)\text{ formal power series }: F(p)=Q(p)=0\}$$

for some polynomial $Q$. One the other hand Artin's Comparison Theorem preserves the differential map $\partial: \Delta^n\to \Delta^{n-1}$ so since $\alpha\in H^{n, et}(\tilde{X}_{loc}^{et}; \mathbb{Z}/p\mathbb{Z})$ by $(\ast)$ one can write
$$\alpha=\{p(t): F(p)=Q(p)=0\}\cap \ker\partial/Im \partial$$

where $\ker \partial$ would cut out linear hyperplanes from a set defined by algebraic power series. Hence $\alpha$ consists of algebraic power series since $\Delta^n$ is an algebraic set. It suffices to remark therefore that $\alpha$ is a Nash set. By assumption it is irreducible. By Bilski \cite[Theorem 2.1]{Bilski} any irreducible Nash set over the complex numbers is an irreducible component of an algebraic set so $\alpha$ is an irreducible component of an algebraic set. Since it is non-singular it is in fact an algebraic set. 

\item Another approach to the Hodge Conjecture is to consider the simplicial cohomology groups
$$H^k(X)=\ker \partial^{\ast}/ \text{Im} \partial^{\ast}$$
where $\partial^{\ast}: \Delta^{\ast}(X)\to \Delta^{\ast}(X)$ is given by $\partial^{\ast}(f)(\Delta)=f(\partial\Delta)$. The idea is first to replace the simplices if dimension $k$ with algebraic sets homeomorphic to $k$-dimennsional spheres. This is possible and does not affect homology except in dimension zero so we assume $k\neq 0$. Indeed it is a question of smoothening the corners and using Stone-Weierstraß. This operation only affects what happens in dimension $k-2$ hence does not affect $H_k$. We will write $\tilde{\Delta}$ for the smoothed simplex and $\Delta\to \tilde{\Delta}$ for the operation.
One then considers instead singular cohomology groups. Then the simplices are replaces by the free abelian group $C_k(X)$ of continuous maps $\sigma_k: \tilde{\Delta}_k\to X$ from the smoothed simplices. The boundary maps are now given by
$\partial \sigma_k (\tilde{\Delta}_k)=\sigma_k(\tilde{\partial\Delta}_k)$. 
We now use Stone-Weierstraß to replace a basis $\sigma_{k, \alpha}$ for $C_k(X)$ of continuous functions by polynomials $f_{k, \alpha}$ converging pointwise for each $\alpha$ to $\sigma_{\alpha}$. To prove the Hodge conjecture is then reduced to proving that this does not affect the cohomology of $X$. First of all one does not need an infinite amount of equations of continuous functions to describe the generators $\alpha_1,\dots, \alpha_k$ of $H^{\ast}(X)$ because if that would have been the case using Stone-Weierstraß to approximate the equations of 
$$\alpha_i=\{\sigma_{k, \alpha}(\partial \Delta_k)=0 \quad \forall \alpha\}/\text{Im} \partial\cap X$$

with polynomials would result in a cycle $\alpha_i'$ with Krull dimension $\dim \alpha_i'=0$. But this is impossible if the approximation is small enough. Secondly we claim that using Stone-Weierstraß and replacing $\sigma_{k, \alpha}$ by polynomials $f_{k, \alpha, \epsilon}$ does not affect cohomology. If not then the result would be a cycle $\alpha'$ non-homologuous to $\alpha$. If the perturbation affects $\sigma_{k, l}(\Delta_k)\subset X$ then a new space $X'$ (since we are perturbing a subset of $X$ we are perturbing $X$) would result but we claim that that space would be diffeomorphic to $X$. This proves that one can use Stone-Weierstraß to replace $C_k(X)$ with the set $P_k(X)$ of polynomials $f_k: \tilde{\Delta}_k\to X$. Then the resulting cycles would all be algebraic and thus settling the conjecture.

It therefore suffices to prove the claim. The simplices are unaffected by the perturbation yet $X$ is. Since $X$ is algebraic the perturbation is that of coefficients of its defining equations. By Sards theorem the set of critical values is non-dense so a small perurbation will not affect $X$ since it is smooth. This establishes the claim and the conjecture. 
\end{enumerate}

\bibliographystyle{plain}
\bibliography{main.bib}

\end{document}